\documentclass[a4paper]{article}

\usepackage{multirow}
\usepackage[english]{babel}
\usepackage[utf8x]{inputenc}
\usepackage[T1]{fontenc}
\usepackage{latexsym,amssymb,amsmath,amsfonts,amscd} 
\usepackage{ulem}
\usepackage{amsmath}
\usepackage[pdftex]{color,graphicx}
\usepackage[colorinlistoftodos]{todonotes}
\usepackage[colorlinks=true, allcolors=blue]{hyperref}

\newtheorem{thm}{theorem}[section]
\newtheorem{theorem}[thm]{Theorem}
\newtheorem{proposition}[thm]{Proposition}
\newtheorem{lema}[thm]{Lemma}
\newtheorem{corollary}[thm]{Corollary}
\newtheorem{remark}[thm]{Remark}
\newtheorem{example}[thm]{Example}
\newtheorem{definition}[thm]{Definition}

\newenvironment{proof}[1][Proof]{\noindent\textbf{#1\,} }{\hfill \rule{0.5em}{0.5em}\medskip}

\title{A note on $\mathbb{Z}$-gradings on the Grassmann algebra and Elementary Number Theory}
\author{Alan Guimar\~aes \\
	Departamento de Matemática\\
	Universidade Federal do Rio Grande do Norte\\
	Natal, RN, 59078-970, Brazil\\
	alansimoes10@hotmail.com\\
	\\
	Claudemir Fidelis\thanks{ {\bf Corresponding Author}}~\thanks{Supported by FAPESP grant No.~2019/12498-0}\\ 
	Unidade Acadêmica de Matemática\\
	Universidade Federal de Campina Grande\\
	Campina Grande, PB 58429-970, Brazil
		\\ and \\Instituto de Matem\'atica e Estat\'istica\\
		USP, 05508-090 S\~ao Paulo, SP, Brazil\\
claudemir.fidelis@professor.ufcg.edu.br\\
\\
	Plamen Koshlukov\thanks{Partially supported by FAPESP grant No.~2018/23690-6 and by CNPq grant No.~302238/2019-0}\\ 
	Department of Mathematics\\
	UNICAMP, 13083-859 Campinas, SP,  Brazil\\
	plamen@unicamp.br
} 

\date{}
\begin{document}
\maketitle

\begin{abstract}
	Let $E$ be the Grassmann algebra of an infinite dimensional vector space $L$ over a field of characteristic zero. In this paper, we study the $\mathbb{Z}$-gradings on $E$ having the form $E=E_{(r_{1},r_{2}, r_{3})}^{(v_{1},v_{2}, v_{3})}$, in which each element of a basis of $L$ has $\mathbb{Z}$-degree $r_{1}$, $r_{2}$, or $r_{3}$. We provide a criterion for the support of this structure to coincide with a subgroup of the group $\mathbb{Z}$, and we describe the graded identities for the corresponding gradings. We strongly use Elementary Number Theory as a tool, providing an interesting connection between this classical part of Mathematics, and PI Theory. Our results are generalizations of the approach presented in \cite{alan-brandao -claud}.

 
 \medskip
 
\end{abstract}

\section{Introduction}

The description of all possible gradings on an algebra $A$ by a group is an important problem in the structure theory of graded rings and its applications. The gradings on ample classes of algebras have already been classified; these include the full matrix algebras, the upper triangular matrix algebras, block-triangular matrix algebras. Assuming that $G$ is a finite abelian group, there is a well-known duality between $G$-gradings and $G$-actions on $A$. In order for the duality to hold one needs a field of characteristic 0 containing a primitive $|G|$-th root of unity. Since $G$ is finite abelian, then $G$ is isomorphic to its dual group, denoted by $G^\ast$, that is the group of all irreducible characters of $G$. Given a grading on $A$, one defines the action through the
dual of $G$ and vice versa. We refer the reader to \cite[Section 3.3, pp. 65–69]{gzbook}, for further details concerning the above duality. 

In this paper, we deal with a specific algebra, namely the Grassmann algebra, graded by the infinite cyclic group $\mathbb{Z}$.

Let $F$ be an infinite field of characteristic different from 2, and let $L$ be a vector space over $F$ with a fixed basis $e_{1}$, $e_{2}$, \dots. The infinite dimensional Grassmann algebra $E$ of $L$ over $F$ is the vector space with a basis consisting of 1 and all products $e_{i_1}e_{i_2}\cdots e_{i_k}$ where $i_{1}<i_{2}<\ldots <i_{k}$, $k\geq 1$. The multiplication in $E$ is induced by $e_{i}e_{j}=-e_{j}e_{i}$ for all $i$ and $j$. We shall denote the above canonical basis of $E$ by $B_{E}$. The Grassmann algebra has a natural $\mathbb{Z}_2$-grading $E_{can}=E_{(0)}\oplus E_{(1)}$, where $E_{(0)}$ is the vector space spanned by 1 and all products $e_{i_1}\cdots e_{i_k}$ with even $k$ while $E_{(1)}$ is the vector space spanned by the products with odd $k$. It is well known that $E_{(0)}$ is the centre of $E$ and $E_{(1)}$ is the ``anticommuting'' part of $E$. Recall that the theory developed by Kemer \cite{basekemer2} has been shaping much of the theory of PI algebras during the last three decades. In Kemer's theory the Grassmann algebra $E$ together with its natural grading, plays a crucial role in it. When the ground field is of characteristic 2, the Grassmann algebra is commutative. 

In recent years a substantial number of papers has presented results on gradings and graded identities on the Grassmann algebra. In the vast majority of them, two conditions have been assumed, namely: (1) the grading group is finite; (2) the generators $e_1$, $e_2$, \dots, $e_m$, \ldots{} of $E$ are homogeneous. This means that the basis of the vector space $L$ is homogeneous. The graded identities for the natural (or canonical) grading $E_{can}$ are well known, see for example \cite{GMZ}. More generally, in the papers \cite{centroneP, disil, LFG}, all homogeneous $\mathbb{Z}_{2}$-gradings on $E$ and the respective graded identities were studied. (Here and in what follows $\mathbb{Z}_2$ stands for the cyclic group of order 2.) Furthermore, in \cite{centroneG, disilplamen, GFD}, gradings and graded identities on $E$ by finite abelian groups of order larger than 2 were also investigated.  
Gradings on $E$ by infinite groups were studied for the first time in \cite{aagpk}. In that paper, the authors presented three types of  $\mathbb{Z}$-gradings on $E$, denoted respectively by $E^{\infty}$, $E^{k^\ast}$ and $E^{k}$, for every non-negative integer $k$. 
These structures are the most canonical $\mathbb{Z}$-gradings on $E$ due to its close relation with the superalgebras $E_{\infty}$, $E_{k^\ast}$ and $E_{k}$ studied in \cite{disil}. It should be noted that there are various other gradings on $E$ by the group $\mathbb{Z}$.

Considering a $G$-grading on $E$, the support of the grading is defined as  the set of the elements $g\in G$ such that there is at least one non-zero element of $E$ of degree $g$. Thus the support of $E^{k^\ast}$ is $\{0,1,\ldots, k\}$ while $E^{\infty}$ and $E^{k}$ have the set $\{0,1,\ldots\}$ of the non-negative integers as support. Furthermore in the paper \cite{alan-brandao -claud}, the authors dealt with $\mathbb{Z}$-gradings on $E$ of full support; they described the $\mathbb{Z}$-graded identities for these structures providing a necessary and sufficient condition for two structures of this type to be isomorphic as $\mathbb{Z}$-graded algebras. The authors of the latter paper pointed out that the support of a grading on $E$ can have interesting properties as a subset of $\mathbb{Z}$. Recently, in \cite{Alan A G}, the first named author of the present paper investigated gradings and graded identities by the additive group of the real numbers, starting a study on gradings by groups that may even be non-enumerable.

We present the general method of constructing homogeneous gradings on the Grassmann algebra $E$. Let $(G, \cdot)$ be a group and $n$ a positive integer. Suppose there exist pairwise commuting $g_1$, \dots, $g_n$ in $G$, where $g_{i}\neq g_{j}$, for all $i$, $j\in\{1,\ldots, n\}$. Additionally, fix $v_{1}$, \dots, $v_{n}$, where $v_{j}\in\mathbb{N}\cup\{\infty\}$, for $1\leq j\leq n$. Let 
 \[
 L=L_{g_{1}}^{v_{1}}\oplus\cdots\oplus L_{g_{n}}^{v_{n}}
 \]
 be a decomposition of $L$ in a direct sum of $n$ subspaces such that $v_{j}=\dim L_{g_{j}}^{v_{j}}$, and assume that the generators $e_i$ of $L$ satisfy $e_i\in \cup_{j=1}^n L_{g_j}^{v_j}$. We define the degree of a generator $e_{k}\in L$ as follows:
 \[
 \|e_{k}\|=g_{j}\quad \mbox{\textrm{ if and only if }}\quad  e_{k}\in L_{g_j}^{v_j}.
 \]
We extend the degree for all monomial $e_{k_1}e_{k_2}\cdots e_{k_s}$ by
 \[
 \|e_{k_1}e_{k_2}\cdots e_{k_s}\|=\|e_{k_1}\|\cdot\|e_{k_2}\|\cdots \|e_{k_s}\|,
 \]
and it gives us a structure of $G$-grading on $E$, denoted by $E_{(g_{1},\ldots, g_{n})}^{(v_{1},\ldots, v_{n})}$. This grading is called \textsl{$n$-induced} $G$-grading on $E$. The elements $g_{1}$, \dots, $g_{n}$ are called \textsl{lower indices} and $v_{1}$, \dots, $v_{n}$ are the \textsl{upper indices} of the grading. We shall denote by $S_{(g_{1},\ldots, g_{n})}^{(v_{1},\ldots, v_{n})}$ the support of the $G$-grading $E_{(g_{1},\ldots, g_{n})}^{(v_{1},\ldots, v_{n})}$.
 
 A complete study of the group graded identities of $E$ is still far from being understood. In the light of the above discussion it is an interesting problem to investigate more closely the structure of the graded identities for the Grassmann algebra. The description of such  identities employs various techniques. One may consider multilinear polynomial identities since they determine all identities of a given algebra in characteristic 0, see \cite{disil, disilplamen}. One of the main tools in this case is  the representation theory of the symmetric and of the general linear groups, and refinements, see for a detailed account the monographs \cite{drbook, gzbook}. When the field $K$ is infinite, one has to consider multihomogeneous identities, for more details of this technique we recommend \cite[Section 4.2]{drbook}. The methods one uses in this case are mostly based on Combinatorial Algebra \cite{centroneP,GFD,aagpk}. A technique that has been used recently is based on Elementary Number Theory. We recall that such an approach appeared implicitly in \cite{disilplamen}, and more explicitly in \cite{CP, alan-brandao -claud}. Although elementary as a method, the  results obtained suggest that a good understanding of such a relation between Number theory and PI theory could give us some insights for the nature of the gradings and graded identities not only on $E$, but also in the case of other important algebras.
  
 In this paper, we shall study gradings on $E$ by the infinite cyclic group $\mathbb{Z}$. Throughout, following the line of research presented in \cite{alan-brandao -claud}, we  deal with the structures of type $E_{(g_{1},\ldots, g_{n})}^{(v_{1},\ldots, v_{n})}$, where $G=\mathbb{Z}$. In this paper we consider initially the case $n=3$. Our results are related to the support $S_{(g_{1},\ldots, g_{n})}^{(v_{1},\ldots, v_{n})}$ and its connection with the structure of the algebra $E_{(g_{1},\ldots, g_{n})}^{(v_{1},\ldots, v_{n})}$. We provide a criterion for a such structure to be of full support, and we describe the respective graded identities. Finally, we discuss grading and graded identities of the $n$-induced $\mathbb{Z}$-gradings on $E$ for $n>3$. More precisely, we will explicit the relationship between the graded identities by the group $\mathbb{Z}$ and $\mathbb{Z}_d$, for a fixed $d\in\mathbb{Z}$. To this end, we employ essentially Elementary Number Theory as a tool, providing an interesting connection between this area and PI-Theory. 
 
 We hope that our results about the 3-induced  $\mathbb{Z}$-gradings on $E$ may shed additional light on the $n$-induced  $\mathbb{Z}$-gradings, for every $n>3$, and consequently on gradings by an arbitrary cyclic group.

\section{Preliminaries}\label{Preliminares}

To simplify both the exposition and the notation, in this paper, we consider the ground field $F$ of characteristic zero. It should be mentioned that several of the results obtained in Subsection  \ref{sub characterization} do not depend on the field $F$, and are characteristic-free. Throughout all algebras will be assumed associative and unitary. The algebras and vector spaces will be always over $F$. We also fix a group $G$.

In this section we recall briefly the tools we use in what follows.

A $G$-grading on an associative algebra $A$ is a vector space decomposition $A=\oplus_{g\in G}A_{g}$ such that  $A_{g}A_{h}\subset A_{gh}$ for every $g$, $h\in G$. Given a non-zero element $a\in A_{g}$ we say that $a$ is homogeneous and its $G$-degree is $\|a\|=g$. The support of a grading is defined as the set $Sup(A)=\{g\in G\mid A_{g}\neq 0\}$. A subspace $B\subset A$ is graded, or homogeneous if $B=\oplus_{g\in G}(B\cap A_{g})$. Similarly, one can define graded subalgebras and graded ideals. Let $A=\oplus_{g\in G}A_{g}$ and $B=\oplus_{g\in G}B_{g}$ be $G$-graded algebras and let $\varphi\colon A\to B$ be an algebra homomorphism. Then $\varphi$ is a \textsl{$G$-homomorphism} (or a graded homomorphism)  whenever $\varphi(A_{g})\subseteq B_{g}$, for every $g\in G$. This means $\varphi$ respects the gradings on $A$ and $B$. Analogously $ \varphi $ is a $G$-isomorphism when it is a $G$-homomorphism and an algebra isomorphism. If $H$ is a subgroup of $G$ we define the quotient $G/H$-grading on $A$ in a natural way as $A=\oplus_{\overline{g}\in G/H}A_{\overline{g}}$.  Here $A_{\overline{g}}=\oplus_{h\in H}A_{gh}$ is the direct sum of the homogeneous component of the $G$-grading on $A$ corresponding to the $H$-coset defined by $g$.

Next we recall the definition of a free $G$-graded algebra. Let $ X=\cup_{g\in G}X_{g}$ be the disjoint union of infinite countable sets of variables $X_{g}=\{x_{1}^{g},x_{2}^{g},\ldots\}$, $g\in G$.  Assuming that for each $g\in G$ the elements of the set $X_{g}$ are of $G$-degree $g$, the free associative algebra $F\langle X|G\rangle$ has a natural $G$-grading $\oplus_{g\in G} F\langle X|G\rangle_{g}$. Here $F\langle X|G\rangle_{g}$ is the vector subspace of $F\langle X|G\rangle$ spanned by all monomials of $G$-degree $g$. The algebra $F\langle X|G\rangle$ is free in the following sense. Given a $G$-graded algebra $A=\oplus_{g\in G} A_g$ and a map $h\colon X\to A$ such that $h(X_g)\subseteq A_g$ for every $g$ then $h$ can be uniquely extended to a $G$-homomorphism $\varphi\colon F\langle X|G\rangle\to A$. An ideal $I$ of $F\langle X|G\rangle$ is said to be a $T_{G}$-ideal if it is invariant under all $G$-endomorphisms $\varphi\colon F\langle X|G\rangle\to  F\langle X|G\rangle$. A polynomial $f(x_{1},\ldots, x_{n})$ in $F\langle X|G\rangle$ is called \textsl{graded polynomial}. The $n$-tuple
$(a_{1},\ldots, a_{n})$ such that $a_{i}\in A_{\|x_{i}\|}$ for $i=1$, \dots, $n$, is called \textsl{$f$-admissible substitution} (or simply admissible substitution), and we say that $f(x_{1},\ldots, x_{n})$ is a \textsl{graded polynomial identity} of $A$ if $f(a_{1},\ldots, a_{n})=0$ for each $f$-admissible substitution. We denote by $T_{G}(A)$ the set of all $G$-graded polynomial identities for $A$. It is easy to show that $T_{G}(A)$ is a $T_{G}$-ideal of $F\langle X|G\rangle$. The converse also holds: every $T_G$-ideal is the ideal of the $G$-graded identities for a suitable $G$-graded algebra. It is well known that whenever $A$ is a $G$-graded algebra over a field of characteristic 0 then $T_{G}(A)$ is generated, as a
$T_G$-ideal, by its multilinear polynomials. Over an infinite field one has to take into account the multihomogeneous polynomials instead of the multilinear ones. If $A$ and $B$ are $G$-graded algebras we say that $A$ and $B$ are PI-equivalent as $G$-graded algebras if $T_{G}(A)=T_{G}(B)$.  Given a $T_{G}$-ideal $I$ of $F\langle X|G\rangle$, the \textsl{variety of $G$-graded algebras} $\mathfrak{V}^{G}$ associated to $I$ is the class of all $G$-graded algebras $A$ such that $I$ is contained in $T_{G}(A)$. In this case, the $T_G$-ideal $I$ is denoted by $T_{G}(\mathfrak{V}^G)$. Finally, we say that the variety $\mathfrak{V}^G$ is generated by the $G$-graded algebra $A$ if $T_{G}(\mathfrak{V}^G)=T_{G}(A)$. 

Now we return to the Grassmann algebra.

\begin{definition}
	Given a monomial $w=e_{i_1}\cdots e_{i_k}$ in $E$, the set $supp (w)=\{e_{i_1},\ldots , e_{i_k}\}$ is called the \textsl{support} of $w$. If $w_{1}=e_{j_1}\cdots e_{j_l}$ is in $E$, then $w$ and $w_{1}$ have \textsl{pairwise disjoint supports} if $supp (w)\cap supp (w_{1})=\emptyset$. In this case we have that $ww_{1}\neq 0$. The length of $w$ is the cardinality of $supp (w)$, and we denote it by $|w|$. 
\end{definition}

Group gradings on the Grassmann algebra have been studied in several papers, see \cite{centroneG, disil, Alan A G,alan-brandao -claud, aagpk}. There are two important types of gradings on $E$.

\begin{definition}\label{deantedez2}
	A $G$-grading $E=\oplus_{g\in G}E_{g}$ on the Grassmann algebra is said to be:
	\begin{enumerate}
		\item [$(a)$] A homogeneous grading if each $e_{i}\in L$ is homogeneous in the grading.
		\item [$(b)$] A grading of full support if its support is the whole of the group $G$.
	\end{enumerate}
\end{definition}

Below we recall the construction of the homogeneous $\mathbb{Z}_{2}$-gradings on $E$. It was given for the first time in \cite{disil}. All homogeneous $\mathbb{Z}_{2}$-gradings on $E$ are given by the following possibilities of $\mathbb{Z}_{2}$-degree on the basis of $L$:

\begin{align*}
\|e_{i}\|_{k} & =\begin{cases} 0,\text{ if }  i=1,\ldots,k\\ 1, \text{ otherwise } 
\end{cases},\\	
\|e_{i}\|_{k^\ast} & =\begin{cases} 1,\text{ if }  i=1,\ldots,k\\ 0, \text{ otherwise } 
\end{cases},\\
\|e_{i}\|_{\infty} & =\begin{cases} 0,\text{ for }  i\text{ even }\\ 1, \text{ for i odd } 
\end{cases}.
\end{align*}
	
We then induce the $\mathbb{Z}_2$-grading on $E$ by putting 
\[
\|e_{j_{1}}\cdots e_{j_{n}}\| =\|e_{j_{1}}\|+\cdots + \|e_{j_{n}}\|,
\]
and extending it to $E$ by linearity. These gradings are denoted by $E_{k}$, $E_{k^\ast}$ and $E_{\infty}$, respectively. When $\|e_{i}\|=1$ for all $i$, we denote the respective $\mathbb{Z}_{2}$-grading by $E_{can}$, which is the natural (or canonical) grading on $E$. We draw the readers' attention that if $L$ is homogeneous (and $G=\mathbb{Z}_2$), then one can always choose a basis of $L$ which is homogeneous in the grading. 

In \cite{disilplamen}, the authors started the study of gradings and graded identities on $E$ by a cyclic group of prime order (different from two). They reduced the study of graded identities to only two cases, which are as follows.
\begin{equation}\label{inf}
\|e_{i}\|_{\infty}=\begin{cases} 0,\text{ if }  i\equiv 0\mod p\\ 1,\text{ if } i\equiv 1\mod p \\
\vdots\\
p-1, \text{ if }i\equiv p-1\mod p \\
\end{cases},
\end{equation}
and
\begin{equation}\label{l_1l_2}
\|e_{i}\|_{l_1,\ldots,l_{p-1}}=\begin{cases} 1,\text{ if }  i=1,\ldots,l_1\\
2,\text{ if }  i=l_1+1,\ldots,l_1+l_2\\
\vdots\\
p-1,\text{ if }  i=l_1+\ldots+l_{p-2}+1,\ldots,l_1+\ldots+l_{p-1}\\
0, \mbox{ otherwise }.
\end{cases}.\end{equation}
The $\mathbb{Z}_p$-degree of the element $e_{i_1}\cdots e_{i_k}$ is given by $\|e_{i_1}\cdots e_{i_k}\|=\|e_{i_1}\|+\cdots +\|e_{i_k}\|$, where the sum is taken modulo $p$. In the same paper the respective bases for graded polynomials identities and the graded codimensions were described. 

Later on in this paper we prove that the graded identities of several $\mathbb{Z}$-gradings on $E$ can be obtained via the graded identities of gradings on $E$ by a finite cyclic group.

We recall the construction of the homogeneous  $\mathbb{Z}$-gradings on $E$ that were introduced in \cite{aagpk}, and later on studied in detail in \cite{alan-brandao -claud}. We write $\mathbb{N}=\{1, 2, 3,\ldots\}$ and $\mathbb{N}_{0}=\mathbb{N}\cup \{0\}$.
Let $n\in\mathbb{N}$, $r_{1}<\ldots <r_{n}$, where each $r_{j}\in\mathbb{Z}$, and $v_{1}$, \dots, $v_{n}$ are such that $v_{j}\in\mathbb{N}$ or $v_{j}=\infty$, for $1\leq j\leq n$. Consider 
\[
L=L_{r_{1}}^{v_{1}}\oplus\cdots\oplus L_{r_{n}}^{v_{n}},
\]
a decomposition of $L$ into $n$ subspaces such that $v_{j}=\dim L_{n_{j}}^{v_{j}}$. Clearly we can assume that the generators $e_i$ of $L$ satisfy $e_i\in \cup_{j=1}^n L_{n_j}^{v_j}$. In other words we split the basis $e_1$, $e_2$, \dots{} of $L$ into $n$ disjoint sets (the bases of the $L_{n_j}$). Given a generator $e_{k}\in L$, we define 
\[
\|e_{k}\|=r_{j}\quad \mbox{\textrm{ if and only if }}\quad  e_{k}\in L_{n_j}^{v_j}.
\]
We extend the degree to all monomials $e_{k_1}e_{k_2}\cdots e_{k_s}$ by
\[
\|e_{k_1}e_{k_2}\cdots e_{k_s}\|=\|e_{k_1}\|+\|e_{k_2}\|+\cdots +\|e_{k_s}\|.
\]
This provides a structure of a $\mathbb{Z}$-grading on $E$, denoted by $E_{(r_{1},\ldots, r_{n})}^{(v_{1},\ldots, v_{n})}$. In this case we say that $E_{(r_{1},\ldots, r_{n})}^{(v_{1},\ldots, v_{n})}$ is the \textsl{$n$-induced} $\mathbb{Z}$-grading. We say that $r_{1}$, \dots,  $r_{n}$ are the \textsl{lower indices} and $v_{1}$, \dots, $v_{n}$ are the \textsl{upper indices} of the grading. 

Recall that we denote by $S_{(r_{1},\ldots, r_{n})}^{(v_{1},\ldots, v_{n})}$ the support of the above grading.

In the paper \cite{alan-brandao -claud}, the authors gave a criterion for the support of a $2$-induced $\mathbb{Z}$-grading on $E$ to be a subgroup of $\mathbb{Z}$. It should be mentioned that such a criterion was obtained by using in an essential way elementary Number theory, namely the use of properties of the greatest common divisor. More precisely the following theorem was proved.

\begin{theorem}\cite[Theorem 3.1]{alan-brandao -claud}\label{o caso dois indices}
$S_{(r_{1}, r_{2})}^{(v_{1}, v_{2})}=\langle d\rangle$ if and only if $v_{1}=v_{2}=\infty$, $r_{1}<0< r_{2}$ and $\gcd(r_{1}, r_{2})=d$. 
\end{theorem}
The theorem above provides the conditions satisfied by the indices of $E_{(r_{1}, r_{2})}^{(v_{1}, v_{2})}$ that guarantee $S_{(r_{1}, r_{2})}^{(v_{1}, v_{2})}$ is a subgroup of $\mathbb{Z}$. 

In this paper we give a partial criterion for a general $n$-induced $\mathbb{Z}$-grading on $E$ to be of full support. It gives conditions on the lower indices of the grading. In the next section we present a similar approach for the $3$-induced $\mathbb{Z}$-gradings on $E$.

\section{Full support: a criterion in terms of Number theory}\label{sub characterization}

As mentioned earlier, several of the results obtained in this section are characteristic-free, the only restriction we impose that the characteristic of the base field cannot be equal to 2. In fact the only statement that requires a field of characteristic 0, is Proposition \ref{Hcentral geral-prop}. we do not claim it is an essential restriction but our proof relies on the base field being of characteristic 0. The validity of Proposition \ref{Hcentral geral-prop} over other fields remains open. 

\begin{lema}
Let $r_{1}$, \dots, $r_{n}\in\mathbb{Z}$, where $n\geq 2$. Then $S_{(r_{1},\ldots, r_{n})}^{(v_{1},\ldots, v_{n})}\leq \mathbb{Z}$ if and only if $S_{(r_{1},\ldots, r_{n})}^{(v_{1},\ldots, v_{n})}=\langle d\rangle $, where $d=\gcd(r_{1},\ldots, r_{n})$.
\end{lema}
\begin{proof}
If $S_{(r_{1},\ldots, r_{n})}^{(v_{1},\ldots, v_{n})}\leq \mathbb{Z}$ then there exists $d'\in\mathbb{Z}$, with $d'>0$, such that $S_{(r_{1},\ldots, r_{n})}^{(v_{1},\ldots, v_{n})}=\langle d' \rangle$. As $r_{1},\ldots, r_{n}\in \langle d' \rangle$, we have that $d'$ divides every one of $r_{1}$, \dots, $r_{n}$. On the other hand, since $d'\in S_{(r_{1},\ldots, r_{n})}^{(v_{1},\ldots, v_{n})}$, we 
obtain that there exist $\alpha_{1}$, \dots, $\alpha_{n}\in\mathbb{N}_{0}$ satisfying
\[
\alpha_{1}r_{1} + \cdots + \alpha_{n}r_{n}=d',
\]
and hence $d$ divides $d'$. Therefore $d'=d$ and the proof is complete.
\end{proof}

Now let $n=3$ and suppose $S_{(r_{1},r_{2}, r_{3})}^{(v_{1},v_{2}, v_{3})}$ is a subgroup of $\mathbb{Z}$. Then it follows that up to the order of the lower indices, we have $r_{1}<0<r_{3}$, $v_{1}=v_{3}=\infty$, and $S_{(r_{1},r_{2}, r_{3})}^{(v_{1},v_{2}, v_{3})}=\langle d\rangle$, where $d=\gcd(r_{1},r_{2}, r_{3})$. It remains to describe the conditions that  $r_{2}$ and $v_{2}$ must satisfy. The following lemma is a standard fact of elementary Number theory and we record it for future use. Its proof is easy and we omit it.

\begin{lema}\label{basico mas útil}
	Let $r_{1}$, $r_{2}$, $r_{3}\in\mathbb{Z}$, and suppose that $d'=\gcd(r_{1}, r_{3})$. Then $\gcd(d', r_{2})=\gcd(r_{1}, r_{2}, r_{3})$. In particular, if $\gcd(r_{1}, r_{2}, r_{3})=1$, then $\overline{r_{2}}$ generates the group $\mathbb{Z}_{d'}=\mathbb{Z}/\langle d'\rangle$. \hfill \rule{0.5em}{0.5em}\medskip
\end{lema}

\begin{example}
	Consider the graded algebra $E_{(-3, 1, 6)}^{(\infty, 2, \infty)}=\oplus_{r\in \mathbb{Z}}A_{r}$, and assume $\|e_{1}\|=\|e_{2}\|=1$. We claim that this grading on $E$ is of full support. In fact, every integer can be written as either $z_{1}=3q$ or $z_{2}=3q + 1$, or $z_{3}=3q + 2$. Clearly one has that $z_{1}\in S_{(-3, 1, 6)}^{(\infty, 2, \infty)}$. Moreover, note that $e_{1}A_{3q}\subset A_{3q + 1}$ and $e_{1}e_{2}A_{3q}\subset A_{3q + 2}$. Then $z_{2}$, $z_{3}\in S_{(-3, 1, 6)}^{(\infty, 2, \infty)}$, and therefore $S_{(-3, 1, 6)}^{(\infty, 2, \infty)}=\mathbb{Z}$.
\end{example}

The example above can be generalized. 
\begin{theorem}\label{caracteriz particular}
	Let $r_{1}$, $r_{2}$, $r_{3}\in\mathbb{Z}$, with $r_{1}<0<r_{3}$ and $\gcd(r_{1}, r_{2}, r_{3})=1$. We have that $S_{(r_{1}, r_{2}, r_{3})}^{(\infty, v_{2}, \infty)}=\mathbb{Z}$ if and only if
	\begin{enumerate}
		\item [(A)] either $S_{(r_{1}, r_{3})}^{(\infty,  \infty)}=\mathbb{Z}$, or
		\item [(B)] $r_{2}\notin S_{(r_{1}, r_{3})}^{(\infty,  \infty)}=\langle d^\prime\rangle$, where $d^\prime=\gcd(r_{1},r_{3})$, and $v_{2}\geq d^\prime -1$.
	\end{enumerate} 
\end{theorem}
\begin{proof}
	Let $r_{1}$, $r_{2}$, $r_{3}\in\mathbb{Z}$, we consider the following two possibilities: (i) $1=\gcd(r_{1},r_{2}, r_{3})=\gcd(r_{1}, r_{3})$, or (ii) $1=\gcd(r_{1},r_{2}, r_{3})< d^{\prime}=\gcd(r_{1}, r_{3})$. If $1=\gcd(r_{1},r_{2}, r_{3})=\gcd(r_{1}, r_{3})$ holds, Theorem~\ref{o caso dois indices} implies that $S_{(r_{1}, r_{2}, r_{3})}^{(\infty, v_{2}, \infty)}=\mathbb{Z}$ is equivalent to statement $(A)$.
	
	Now we assume the conditions of $(B)$ and take $z\in \mathbb{Z}$, and we require that $z>0$. Euclidean division implies that there exist non-negative integers $q$ and $0\leq r < d'$ such that
		\begin{equation}\label{AEC}
		z=qd' +r.
		\end{equation}
	On the other hand, Lemma \ref{basico mas útil} yields that $\overline{r_{2}}$ generates $\mathbb{Z}_{d'}$. Hence, there exists an integer $0\leq u \leq d^{\prime} -1$ such that
	\[\overline{ur_{2}}=\overline{z}=\overline{r}.\]
	It follows that there exists integer $q_{1}$ such that $r=ur_{2} + q_{1}d'$, hence
\[
z=qd' + ur_{2} + q_{1}d'=(q+q_{1})d' + ur_{2}.
\]
	Assume that $U=\{e_{1}, \ldots, e_{v_2}\}$ is the set of all generators of homogeneous degree $r_{2}$. Since $S_{(r_{1}, r_{3})}^{(\infty,  \infty)}=\langle d^{\prime}\rangle\subset S_{(r_{1}, r_{2}, r_{3})}^{(\infty, v_{2}, \infty)}$, we take a monomial $w_{1}$ satisfying $\|w_{1}\|=(q+q_{1})d'$, and $supp(w_{1})\cap U=\emptyset$. Now, as $v_{2}\geq d' -1$, putting $w_{2}=e_{1} \cdots e_{u}$, it follows that 
	\[z=(q+q_{1})d' + ur_{2}=\|w_{1}\|+\|w_{2}\|=\|w_{1}w_{2}\|\in S_{(r_{1}, r_{2}, r_{3})}^{(\infty, v_{2}, \infty)},\]
	 which implies $z\in\mathbb{Z}$. Now let us show that $-z\in S_{(r_{1}, r_{2}, r_{3})}^{(\infty, v_{2}, \infty)}$. There exist a  negative integer $q_{2}$ and $0\leq s < d'$ such that
	 	\[-z = q_{2}d' +  s.
	 	\]
By the previous argument, there exists $0\leq v\leq d'-1$ so that
	 	\[\overline{vr_{2}}=\overline{z}=\overline{s}.\]
	 	It follows that there exists integer $q_{3}$ such that $s=vr_{2} + q_{3}d'$, hence
	 	\[-z=q_{2}d' + vr_{2} + q_{3}d'=(q_{2}+q_{3})d' + vr_{2}.\]
	 	Since $\langle d^\prime\rangle$ is contained in the support, we can take a monomial in $E$ of degree $(q_{2}+q_{3})d'$. Now, repeating the process which was done above, it follows that $-z\in S_{(r_{1}, r_{2}, r_{3})}^{(\infty, v_{2}, \infty)}$, so the grading is of full support.
	
	Now we will show that $S_{(r_{1}, r_{2}, r_{3})}^{(\infty, v_{2}, \infty)}=\mathbb{Z}$ and (ii) $1=\gcd(r_{1},r_{2}, r_{3})< d^{\prime}=\gcd(r_{1}, r_{3})$ imply statement $(B)$. First of all, notice that $r_{2}\notin S_{(r_{1}, r_{3})}^{(\infty, \infty)}$. As $-r_{2}\in S_{(r_{1}, r_{2}, r_{3})}^{(\infty, v_{2}, \infty)}$, there exist $\alpha$, $\beta$, $\gamma$ in $\mathbb{N}_{0}$ satisfying
	\[\alpha r_{1}+ \beta r_{2}+ \gamma r_{3}=-r_{2},\]
	where $0<\beta\leq v_{2}$. Hence
	\[\alpha r_{1}+\gamma r_{3}=(-\beta -1)r_{2},\]
	so $d^{\prime}$ divides $(-\beta -1)r_{2}$. As $\gcd(d^{\prime}, r_{2})=1$, we have that $d^{\prime}$ divides $|(-\beta -1)|=\beta+1>0$. Therefore,
\[	d^{\prime}\leq \beta+1\leq v_{2}+1 ,\]
	and we are done.
\end{proof}

\begin{corollary}\label{caracteriz}
	Let $r_{1}$, $r_{2}$, $r_{3}\in\mathbb{Z}$, with $r_{1}<0<r_{3}$ and $d=\gcd(r_{1}, r_{2}, r_{3})$. We have $S_{(r_{1}, r_{2}, r_{3})}^{(\infty, v_{2}, \infty)}=\langle d\rangle$ if and only if
	\begin{enumerate}
		\item $S_{(r_{1}, r_{3})}^{(\infty,  \infty)}=\langle d\rangle$ or
		\item $r_{2}\notin S_{(r_{1}, r_{3})}^{(\infty,  \infty)}=\langle d^\prime\rangle$, where $d^\prime=\gcd(r_{1},r_{3})$, and $v_{2}\geq d^\prime -1$.
	\end{enumerate} 
\end{corollary} 
\begin{proof}
	The proof is a slight and obvious modification of that of Theorem \ref{caracteriz particular}.
\end{proof}

\begin{definition}
	Let $G$ be an abelian group and $H\leq G$. We say that a $G$-grading $E=\oplus_{g\in G}E_{g}$ is a $H$-central $G$-grading, if
	\begin{itemize}
		\item $E=\oplus_{g\in G}E_{g}$ is homogeneous,
		\item $E=\oplus_{g\in G}E_{g}$ is of full support, and
		\item For each $h\in H$, the component $E_{h}$ has infinitely many monomials of even length with pairwise disjoint supports.
	\end{itemize}
	When $G=\mathbb{Z}$ and $H=\langle d\rangle$, we simply say that the grading is $d$-central. We write $E=E^{c}$ to indicate that $E$ is endowed with a $d$-central $\mathbb{Z}$-grading.
\end{definition}

Let $\pi_{H}\colon F\langle X|G\rangle\to F\langle X|G/H\rangle$ be the homomorphism of free algebras given by
\[\pi_{H}(x_{i}^{g})=x_{i}^{\overline{g}},\] 
where $x_{i}^{g}$ denotes a variable of $G$-degree $g$ while $x_{i}^{\overline{g}}$ is a variable of $G/H$-degree $\overline{g}$. When $G=\mathbb{Z}$ and $H=\langle d\rangle$ the previous homomorphism will be denoted by $\pi_{d}$.

The following proposition is one of the key steps in obtaining our main results. Its proof is based on multilinear polynomials hence one needs a field of characteristic 0.

\begin{proposition}\label{Hcentral geral-prop}
	Let $E^{c}=\bigoplus_{g\in G}A_{g}$ be a $H$-central $G$-grading on $E$ and let $B=\bigoplus_{\overline{g}\in G/H}A_{\overline{g}}$ be the respective $G/H$-grading. Let $I=T_{G/H}(B)$ and $J=\{f\in F\langle X|G\rangle\mid \pi_{H}(f)\in I\}$. Over a field of characteristic zero, we have that $T_{G}(E^{c})= J$.	
\end{proposition}
\begin{proof}
We follow word by word, with small modifications, the idea of either \cite[Corollary 5.7]{alan-brandao -claud} or \cite[Proposition 3.3]{centroneG}. 
\end{proof}

\section{Graded identities for gradings of full support}\label{sect4}

Till the end of the paper, we fix two positive integers $a$ and $b$, and  an integer $c$ such that $-b<c<a$.

We consider the respective decomposition of $L$ of the form
\[
L=L^{v}_{-b}\oplus L^{k}_{c}\oplus L^{u}_{a}
\]
such that $S_{(-b,c, a)}^{(v,k, u)}=\mathbb{Z}$. According to Theorem \ref{caracteriz particular}, we have that $v=u=\infty$ and $k$ is a non-negative  integer satisfying additional conditions. In this section we study the graded identities for these gradings on $E$. 

\subsection{Case $(A)$}

We start with the graded identities of the $3$-induced $\mathbb{Z}$-gradings on $E$ given as in item $(A)$ of Theorem \ref{caracteriz particular}. 

Let $k$ be a positive integer, and consider the $\mathbb{Z}_{2}$-gradings $E_{k}$ and $E_\infty$ given after Definition \ref{deantedez2}. The corresponding graded identities were described in \cite{disil}. Denote by $I_k$ the $T_{2}$-ideal of the $\mathbb{Z}_{2}$-graded identities of $E_{k}$, and by $I_\infty$ $T_{2}$-ideal of the $\mathbb{Z}_{2}$-graded identities of $E_{\infty}$. We denote by
\begin{itemize}
	\item $\mathfrak{V}_{(-1, 0, 1)}^{(\infty, k,\infty )}$ the variety of $\mathbb{Z}$-graded algebras defined by the polynomials
	\[
	\{f\in F\langle X|\mathbb{Z}\rangle\mid \pi_{2}(f)\in I_{k}\},
	\]
	\item $\mathfrak{V}_{(-1, 0, 2)}^{(\infty, k,\infty )}$ the variety of $\mathbb{Z}$-graded algebras defined by the polynomials
	\[
	\{f\in F\langle X|\mathbb{Z} \rangle\mid \pi_{2}(f)\in I_{\infty}\}.
	\]
	\end{itemize}

\begin{theorem}\label{r1r2r3=0}
Let $k$, $a$, $b\in\mathbb{N}$. If $\gcd(a,b)=1$, then 
\begin{enumerate}
\item [(1)] $E_{(-b, 0, a)}^{(\infty, k,\infty )}$ is of full support. 
\item [(2)] $E_{(-b, 0, a)}^{(\infty, k,\infty )}$ generates the variety $\mathfrak{V}_{(-1, 0, 1)}^{(\infty, k,\infty )}$, when the product $a b$ is odd.

\item [(3)] $E_{(-b, 0, a)}^{(\infty, k,\infty )}$ generates the variety $\mathfrak{V}_{(-1, 0, 2)}^{(\infty, k,\infty )}$, when $a b$ is even.
	\end{enumerate}
\end{theorem} 
\begin{proof}
	Statement (1): Since $S_{(-b, a)}^{(\infty,\infty )}= S_{(-b, 0, a)}^{(\infty, k,\infty )}$, we apply Theorem \ref{o caso dois indices}. 

	Statement (2): It suffices to check that $E_{(-b, 0, a)}^{(\infty, k,\infty )}$ is a $2$-central $\mathbb{Z}$-grading. Then it induces the $\mathbb{Z}_{2}$-grading $E_{k}$. Now we apply Proposition \ref{Hcentral geral-prop}.  
	
	Statement (3) is proved using similar arguments.
\end{proof}

\begin{corollary}
	If $\gcd(a,b)=1$, then $E_{(-b, 0, a)}^{(\infty, \infty,\infty )}$ generates $\mathfrak{V}_{(-1, 0, 2)}^{(\infty, k,\infty )}$. In particular, if  the product $a b$ is even and $k\in\mathbb{N}$, the algebras $E_{(-b, 0, a)}^{(\infty, \infty,\infty )}$ and $E_{(-b, 0, a)}^{(\infty, k,\infty )}$ are PI-equivalent. 
\end{corollary}

Therefore if $k_{1}<k_{2}$ are positive integers then $E_{(-1, 0, 1)}^{(\infty, k_{1},\infty )}$ is $\mathbb{Z}$-isomorphic to a homogeneous subalgebra of $E_{(-1, 0, 1)}^{(\infty, k_{2},\infty )}$. 

Taking into account Theorem \ref{r1r2r3=0}, given positive integers $k_{1}<k_{2}<\ldots$, it follows that
\[
T_{\mathbb{Z}}(\mathfrak{V}_{(-1, 0, 1)}^{(\infty, k_{1},\infty )})\supset T_{\mathbb{Z}}(\mathfrak{V}_{(-1, 0, 1)}^{(\infty, k_{2},\infty )})\supset\ldots.
\]
Therefore the $3$-induced $\mathbb{Z}$-gradings of full support on $E$ provide infinitely many varieties of $\mathbb{Z}$-graded algebras. This is in sharp contrast with the $2$-induced $\mathbb{Z}$-gradings of full support, as can be seen in \cite{alan-brandao -claud}.

Let $\mathfrak{V}_{can}$ be the variety of $\mathbb{Z}$-graded algebras defined by the graded identities

	\begin{enumerate} 
		\item	$[x_{1}, x_{2}]$, if either $\|x_{1}\|$ or $\|x_{2}\|$ is even and
		\item $x_{1}x_{2}+x_{2}x_{1}$, if both $\|x_{1}\|$ and $\|x_{2}\|$ are odd.
	\end{enumerate}

Notice that the variety of $\mathbb{Z}$-graded algebras $\mathfrak{V}_{can}$ can be interpreted as
\[
\{f\in F\langle X|\mathbb{Z}\rangle\mid \pi_{2}(f)\in T_2(E_{can})\}.
\]

\begin{theorem}\label{r1r2r3 dif 0}
	Assuming $\gcd(a,b)=1$ and $c\neq 0$, we have that 
	\begin{enumerate}
		\item [(1)] $E_{(-b, c, a)}^{(\infty, k,\infty )}$ is a $2$-central $\mathbb{Z}$-grading. 
		\item [(2)] $E_{(-b, c, a)}^{(\infty, k,\infty )}$ generates
		\begin{enumerate}
			\item[(i)] the variety $\mathfrak{V}_{(-1, 0, 1)}^{(\infty, k,\infty )}$, if the product $a b$ is odd and $c$ is even;
			\item[(ii)] the variety $\mathfrak{V}_{(-1, 0, 2)}^{(\infty, k,\infty )}$, if $a b$ is even;
			\item[(iii)] the variety $\mathfrak{V}_{can}$, if $a b c$ is odd.
		\end{enumerate}
\item [(3)] $E_{(-b, c, a)}^{(\infty, \infty,\infty )}$ generates
	\begin{enumerate}
		\item[(i)] the variety $\mathfrak{V}_{(-1, 0, 2)}^{(\infty, k,\infty )}$, if the product $a b c$ is even;
		\item[(ii)] the variety $\mathfrak{V}_{can}$, if $a b c$ is odd.	
	\end{enumerate}
\end{enumerate}
\end{theorem} 
\begin{proof}
Let us prove the first statement. It is known that all 2-induced $\mathbb{Z}$-gradings of full support are $2$-central, see \cite{alan-brandao -claud}. Since  $E_{(-b, a)}^{(\infty,\infty )}$ is a  graded subalgebra of $E_{(-b, c, a)}^{(\infty, k,\infty )}$, the result follows. 

For the second statement, according to the first one, $E_{(-b, c, a)}^{(\infty, k,\infty )}$ is a $2$-central $\mathbb{Z}$-grading. Hence, if $a b$ is odd, there are two possibilities: i)  if $c$ is even then the structure $E_{(-b, c, a)}^{(\infty, k,\infty )}$ induces the $\mathbb{Z}_{2}$-grading $E_{k}$; and ii) if $c$ is odd then the structure $E_{(-b, c, a)}^{(\infty, k,\infty )}$ induces the $\mathbb{Z}_{2}$-grading $E_{can}$. Similarly to Theorem \ref{r1r2r3=0}, the result follows applying Proposition \ref{Hcentral geral-prop}. The remaining cases are treated in a similar way.  
\end{proof}

Theorems \ref{r1r2r3=0} and \ref{r1r2r3 dif 0} give us a complete description of the graded identities for the gradings in case $(A)$.

\subsection{Case $(B)$}

Now we shall deal with the structures presented in item (B) of Theorem \ref{caracteriz particular}. Recall that $a$ and $b$ denote positive integers, and $c$ is an integer satisfying $-b<c<a$. 

\begin{definition}\label{deffiel}
	We say that the set $\{a, b, c\}\subset\mathbb{Z}$ satisfies the condition $\hat{B}$ if
	\begin{itemize}
		\item $\gcd(a,b,c)=1$ and
		\item $\gcd(a,b)$, $\gcd(a,c)$, $\gcd(b,c)\neq 1$.
	\end{itemize}
\end{definition}

\begin{remark}\label{cnaozero}
	Assume that $k\geq d-1$, where $d=\gcd(a,b)$. 
	\begin{itemize}
	\item The set $\{a, b, c\}\subset\mathbb{Z}$ satisfies $\hat{B}$ if and only if $E_{(-b, c, a)}^{(\infty, k,\infty )}$ is as in case $(B)$ of Theorem~\ref{caracteriz particular}.
	\item If $\{a, b, c\}\subset\mathbb{Z}$ satisfies $\hat{B}$ , then $0\notin \{a, b, c\}$.
		\end{itemize} 
\end{remark}

According to Theorem \ref{caracteriz particular}, if $\{a, b, c\}\subset\mathbb{Z}$ satisfies $\hat{B}$, and $k\geq d -1$, then the $\mathbb{Z}$-graded algebra $E_{(-b, c, a)}^{(\infty, k,\infty )}$ is of full support. There are two possibilities that we must take into consideration, namely either $k=\infty$ or $k<\infty$. We first consider the case $k=\infty$.

\begin{lema}\label{muitos de mesmo length}
	Suppose that $E_{(-b, c, a)}^{(\infty,\infty,\infty )}=\oplus_{r\in\mathbb{Z}}A_r$ is of full support. Given a monomial $w\in A_r$, there exist infinitely many monomials with pairwise disjoint supports in $A_{r}$ whose length is equal to the length of $w$.
\end{lema}
\begin{proof}
We can assume, without loss of generality, that
\begin{align*}
L_{-b}^\infty&=span_F\{e_{i}\mid i\equiv 0\pmod{3}\},\\
L_{a}^\infty&=span_F\{e_{i}\mid i\equiv 1\pmod{3}\},\\
L_{c}^\infty&=span_F\{e_{i}\mid i\equiv 2\pmod{3}\},
\end{align*}
	and 
	$w = e_{i_1}\cdots e_{i_l}  e_{j_1}\cdots e_{j_q}e_{h_1}\cdots e_{h_s}\in A_r$, where $e_{i_t}\in L_{-b}^\infty$, $e_{j_u}\in L_{a}^\infty$, and $e_{h_v}\in L_{c}^\infty$, for $1\leq t\leq l$, $1\leq u\leq q$, and $1\leq v\leq s$. Hence for all choices of $l$ elements in $L_{-b}^\infty$, $q$ elements in $L_{a}^\infty$ and $s$ elements in $L_{c}^\infty$, we obtain a monomial in $A_{r}$ whose length is equal to the length of $w$.
\end{proof}

	Assuming that $E_{(-b, c, a)}^{(\infty,\infty,\infty )}$ is of full support, then there exist $\alpha$, $\beta$, $\omega$, $\alpha^\prime$, $\beta^\prime$, and $\omega^\prime$ in $\mathbb{N}_0$ satisfying
\begin{align*}
1&=\alpha a+\beta(-b)+\omega c,\\
-1&=\alpha^\prime a+\beta^\prime(-b)+\omega^\prime c.
\end{align*}
Now we fix $\alpha$, $\beta$, $\omega$, $\alpha'$, $\beta'$ and $\omega'$ satisfying these equalities such that the sums $\alpha +\beta+\omega$ and $\alpha' +\beta'+\omega'$ are minimal. We have to take into account the following four cases:
\begin{enumerate}
	\item [(1)] Both $\alpha + \beta+\omega$ and $\alpha' + \beta'+\omega'$ are even.
	\item [(2)] $\alpha + \beta+\omega$ is even and $\alpha' + \beta'+\omega'$ is odd.
	\item [(3)] $\alpha + \beta+\omega$ is odd and $\alpha' + \beta'+\omega'$ is even.
	\item [(4)] Both $\alpha + \beta+\omega$ and $\alpha' + \beta'+\omega'$ are odd.
\end{enumerate}

\begin{lema}\label{centralinf}
Assume that $E_{(-b, c, a)}^{(\infty,\infty,\infty )}=\bigoplus_{r\in\mathbb{Z}}A_r$ is of full support. If $r$ is even, the component $A_{r}$ has infinitely many elements of $B_E$ of even length, and with pairwise disjoint supports. 
\end{lema}
\begin{proof}
First we consider the $\mathbb{Z}$-grading on $E$ of case $(1)$. 	As $\alpha +\beta+\omega$ is even, there exist infinitely many monomials of even length with pairwise disjoint supports in $A_{1}$, which is also true for $A_{-1}$. Hence if $r\neq 0$ is even, we have either $\underbrace{A_1A_1\cdots A_1}_{|r|}\subseteq A_r$ or $\underbrace{A_{-1}A_{-1}\cdots A_{-1}}_{|r|}\subseteq A_r$. Moreover, $A_{-1}A_{1}\subset A_{0}$, thus the result follows. The same argument can be applied in the case $(4)$.

Now, we consider the case $(2)$. As $\alpha +\beta+\omega$ is even, there exist infinitely many monomials of even length with pairwise disjoint supports in $A_{1}$; also there exist infinitely many monomials of odd length, and with pairwise disjoint supports in $A_{-1}$. It follows that $A_{0}$ has infinitely many monomials of even and of odd lengths with pairwise disjoint supports. Now given a monomial $w\in A_{r}$, due to Lemma \ref{muitos de mesmo length}, we have that $A_{r}$ has infinitely many monomials with pairwise disjoint supports and length equal to that of $w$. Since $wA_{0}\subset A_r$, the last statement follows. Case $(3)$ is dealt with using a similar argument.
\end{proof}

\begin{theorem}
	Assume $E_{(-b, c, a)}^{(\infty,\infty,\infty )}=\bigoplus_{r\in\mathbb{Z}}A_r$ is of full support. Then
	\begin{enumerate}
		\item [(1)] $E_{(-b, c, a)}^{(\infty, \infty,\infty )}$ is a $2$-central $\mathbb{Z}$-grading. 
		\item [(2)] $E_{(-b, c, a)}^{(\infty, \infty,\infty )}$ generates
			\begin{enumerate}
			\item[(i)] the variety $\mathfrak{V}_{(-1, 0, 2)}^{(\infty, k,\infty )}$, if $a b c$  is even;
			\item[(ii)] the variety $\mathfrak{V}_{can}$, if $a b c$ is odd.
		\end{enumerate}
	\end{enumerate}
\end{theorem}
\begin{proof}
Lemma \ref{centralinf}  implies statement $(1)$, so it is enough to prove statement (2). If $a b c$ is even, it follows that in the set $\{a,b,c\}$ there is a pair $(r,s)$ such that one of $r$ and $s$ is even and the other is odd. Hence the induced $\mathbb{Z}_{2}$-grading is $E_{\infty}$. Now, if $a b c$ is odd, we have that $a$, $b$, and $c$ are odd integers, so the induced $\mathbb{Z}_{2}$-grading is $E_{can}$. Finally, in both cases we apply Proposition \ref{Hcentral geral-prop}. 
\end{proof}

Our aim is to complete the description of the graded identities of $E_{(-b, c, a)}^{(\infty, k,\infty )}$, now considering case (B) of Theorem \ref{caracteriz particular}, in which $k$ is finite. At first we develop   some results concerning the support of the $2$-induced grading $E_{(-b, a)}^{(\infty,\infty )}$, where $d=\gcd(a, b)$. By Theorem \ref{o caso dois indices} there exist $\alpha$, $\beta$, $\alpha '$ and $\beta '$ in $\mathbb{N}_{0}$ satisfying 
\[d=\alpha a + \beta (-b), \quad\textrm{ and }\quad -d=\alpha ' a + \beta ' (-b).\]
We fix $\alpha$, $\beta$, $\alpha'$  and $\beta'$ satisfying these equalities such that the sums $\alpha +\beta$ and $\alpha' +\beta'$ are the least possible.

We have to take into account the following four cases:
\begin{enumerate}
	\item [(1$^\prime$)] Both $\alpha + \beta$ and $\alpha' + \beta'$ are even.
	\item [(2$^\prime$)] $\alpha + \beta$ is even and $\alpha' + \beta'$ is odd.
	\item [(3$^\prime$)] $\alpha + \beta$ is odd and $\alpha' + \beta'$ is even.
	\item [(4$^\prime$)] Both $\alpha + \beta$ and $\alpha' + \beta'$ are odd.
\end{enumerate}
The statements below give us information about the structures $E_{(-b, c, a)}^{(\infty,k,\infty )}$ of full support.

\begin{proposition}\label{lemas juntos}
	Suppose that the set $\{a, b, c\}$ satisfies $\hat{B}$, let $d=\gcd(a, b)$ and $k\geq d-1$. If $E_{(-b, c, a)}^{(\infty,k,\infty )}$ is a $\mathbb{Z}$-graded algebra  of cases $(1^\prime)$, $(2^\prime)$, and $(3^\prime)$, then it is $d$-central. 
\end{proposition}
\begin{proof}
	Let $E_{(-b, c, a)}^{(\infty,k,\infty )}=\bigoplus_{r\in\mathbb{Z}}A_r$. Assuming case (1$^\prime$), it follows that both $A_{d}$ and $A_{-d}$ have infinitely many monomials of even length with pairwise disjoint supports. Hence the same holds for every component $A_{r}$, where $r\in\langle d\rangle$.
	
	Assuming now case (2$^\prime$), we conclude that $A_{0}$ has infinitely many monomials of even and odd lengths with pairwise disjoint supports. Given $r\in\langle d\rangle$, if $A_{r}$ has monomials of even length, then it has infinitely many monomials of even length with pairwise disjoint supports. Otherwise, as  $A_{0}A_{r}\subset A_{r}$, we use the monomials of odd length in  $A_{0}$ to obtain the same conclusion. Case (3$^\prime$) is dealt with similarly.
\end{proof}

\begin{proposition}\label{caso (4) d central}
	Suppose that the set $\{a, b, c\}$ satisfies $\hat{B}$, let $d=\gcd(a, b)$, and $k\geq d-1$. If $E_{(-b, c, a)}^{(\infty,k,\infty )}$ is a $\mathbb{Z}$-graded algebra  as in case $(4^\prime)$, then it is $2d$-central. 
\end{proposition}
\begin{proof}
	Let $E_{(-b, c, a)}^{(\infty,k,\infty )}=\bigoplus_{r\in\mathbb{Z}}A_r$. As both $A_{-d}$ and $A_{d}$ contain infinitely many monomials of odd length, given an even integer $n\neq 0$, we conclude that $\underbrace{A_dA_d\cdots A_d}_{|n|}\subseteq A_{nd}$ or $\underbrace{A_{-d}A_{-d}\cdots A_{-d}}_{|n|}\subseteq A_{nd}$, and $A_{-d}A_{d}\subset A_{0}$. 
\end{proof}

\begin{corollary}
	Suppose that $\{a, b, c\}$ satisfies $\hat{B}$, let $d=\gcd(a, b)$, and $k\geq d-1$. Then there exists a non-zero subgroup $H\leq \mathbb{Z}$ such that $T_{\mathbb{Z}}(E_{(-b, c, a)}^{(\infty,k,\infty )})=\{F\langle X|\mathbb{Z}\rangle\mid \pi_{H}(f)\in I_{H}\}$, where $I_{H}$ is the $T_{\mathbb{Z}/H}$-ideal of the respective $\mathbb{Z}/H$-grading on $E$.
\end{corollary}

Actually, concerning the gradings of case (4$^\prime$), we can deduce somewhat more information.

\begin{lema}\label{m, n odd implies type 4}
	If $a$ and $b$ are integers such that $a/d$ and $b/d$ are odd integers, where $d=\gcd(a,b)$, then $E_{(-b, a)}^{(\infty, \infty)}$ is of case $(4)$.
\end{lema}
\begin{proof}
	The proof follows word by word that of \cite[Lemma 3.10]{alan-brandao -claud}. In fact, the equalities
	\[
	1=\alpha (a/d) + \beta (-b/d), \quad\textrm{ and }\quad -1=\alpha '(a/d) + \beta '(-b/d),
	\]
	imply that among $\alpha$ and $\beta$ one is even and the other is odd, and also among $\alpha'$ and $\beta'$. Then both $\alpha +\beta$ and $\alpha' +\beta'$ are odd, hence $E_{(-b, a)}^{(\infty, \infty)}$ is of type $(4)$.
\end{proof}

\begin{lema}\label{A in centro}
	Suppose that $\{a, b, c\}$ satisfies $\hat{B}$, $d=\gcd(a, b)$, and $k\geq d-1$. Assume that $E_{(-b,a)}^{(\infty,\infty )}$ is a $\mathbb{Z}$-graded algebra of case $(4^\prime)$ where $a/d$ and $b/d$ are odd integers. Moreover, let $E_{(-b,c, a)}^{(\infty,k, \infty)}= \oplus_{r\in\mathbb{Z}}A_{r}$ and $E_{can}=E_{(0)}\oplus E_{(1)}$. If $r=(2{l})d$ then $A_r\subset Z(E)$. If $r=(2l-1)d$ then $A_r\subseteq E_{(1)}$.
\end{lema}
\begin{proof}
	Let $r\geq 0$ be an integer. Let
	\[w=e_{i_1}\cdots e_{i_t}e_{j_1}\cdots e_{j_k}\]
	be a monomial of $E$. Assume that $e_{i_{1}}$, \ldots, $e_{i_{t}}$ have $\mathbb{Z}$-degree $a$, while $e_{j_{1}}$, \ldots, $e_{j_{k}}$ are of degree $-b$. For all $r\in\langle d \rangle$, we have 
	\[w\in A_{r}\quad \mbox{\textrm{ if and only if }}\quad r=ta+k(-b).\]
	If $r=(2{l})d$, since  $a/d$ and $b/d$ are odd  and $2l=t(a/d)+k(-b/d)$, we conclude that $t$ and $k$ have the same parity. Thus $t+k$ is even and $A_{r}\subset Z(E)$. If $r=(2{l}-1)d$, we conclude that $t$ and $k$ have different parities, thus $A_{r}\subset E_{(1)}$. The same argument works when we deal with $-r$.
\end{proof} 

Suppose now that the product $(a/d) (b/d)$ is even, and consider the structure $E_{(-b,c, a)}^{(\infty,k, \infty)}=\bigoplus_{r\in\mathbb{Z}}A_{r}$. Hence we have $(a/d)+(b/d)$ is an odd integer. Since 
\[0=a (b/d) + (-b) (a/d),\]
the component $A_{0}$ contains infinitely many monomials of odd length, namely of length $(a/d)+(b/d)$.

As $d=\alpha a +\beta (-b) $ and $-d=\alpha' a +\beta' (-b) $, the components $A_d$ and $A_{-d}$ contain monomials of odd length. 

Hence $A_{0}$ has infinitely many monomials of  even, and also of odd length, with pairwise disjoint supports. Therefore this also holds for each homogeneous component of degree $r\in\langle d\rangle$. In this way we proved the following lemma.

\begin{lema}\label{type4central}
	If $E_{(a, -b)}^{(\infty, \infty)}$ is of type $(4)$, and if the product $(a/d) (b/d)$ is even, then $E_{(-b, c,a)}^{(\infty,k, \infty)}$ is $d$-central.
\end{lema}

We give the explicit form of the graded identities for the gradings of full support on $E$. Here we are considering $U=\cup_{n\in\mathbb{Z}}X_{nd}$. 

\begin{theorem}\label{idGl1l2}
Suppose that $\{a, b, c\}$ satisfies $\hat{B}$, and assume that the product $(a/d) (b/d)$ is  even, where $\gcd(a,b)=d$ and $k\geq d-1$. The $T_\mathbb{Z}$-ideal of the graded identities for  $E_{(-b, c, a)}^{(\infty,k,\infty )}$ is generated by the polynomials 
\begin{enumerate}
\item[(i)] $[x_1, x_2, x_3]$, where each $x_i\in X_\mathbb{Z}$;
\item[(ii)] $u_1u_2\cdots u_{\alpha}$, where the variables $u_i$ lie in $X_\mathbb{Z}\setminus U$ with $\|u_i\|=q_id+r_ic$ such that $1\leq r_i\leq d-1$, $r_1+\cdots+r_\alpha>k$.
\end{enumerate}
\end{theorem}
\begin{proof}
Due to Proposition \ref{lemas juntos} and Lemma \ref{type4central}, we have that the algebra $E_{(-b, c, a)}^{(\infty,k,\infty )}$ is equipped with a $d$-central $\mathbb{Z}$-grading. We denote its induced $\mathbb{Z}_d$-grading by $E_{c,k}^d$. As $\gcd(d,c)=1$, it follows from Lemma \ref{basico mas útil} that $\bar{c}$ is a generator of the group $\mathbb{Z}_d$. Hence, without loss of generality, we can assume $\bar{c}=\bar{1}$. Consider $E_{c,k}^d=G_{\overline{0}}\oplus G_{\overline{1}}\oplus \cdots\oplus G_{\overline{d-1}}$ and let $L_{\overline{t}}=L\cap G_{\overline{t}}$, for each $t=0$, \dots, $d-1$. We conclude that $\dim L_{\bar{0}}=\infty$, $\dim L_{\bar{1}}=k$ and $\dim L_{\bar{i}}=0$, for every $i=2$, \dots, $d-1$. We define $A$ as the set
\[
\{(\alpha_1,\ldots,\alpha_{q-1})\in\mathbb{N}^{d-1}\mid x^{\bar{1}}_{1}\cdots x^{\bar{1}}_{\alpha_1}x^{\bar{2}}_{1}\cdots  x^{\overline{d-1}}_{\alpha_{d-1}}\not\in T_d(E_{c,k}^d)\}.
\]
Since the field $F$ is of characteristic zero, following the same argument that was used in \cite[Proposition 5]{disilplamen}, we can show that if $(\alpha_1,\ldots,\alpha_{d-1})\in A$ then
\[
\psi(V_{\alpha}\cap T(E))=V_{\alpha_{0},\alpha_{1},\ldots, \alpha_{d-1}}\cap T_{d}(E_{c,k}^d),
\]
where $T(E)$ is the ordinary $T$-ideal of $E$ and $\alpha=\alpha_0+\alpha_{1}+\cdots+\alpha_{d-1}$. Now it suffices to apply Proposition \ref{Hcentral geral-prop}.  
\end{proof}

In order to complete the description of the bases for the graded identities for the 3-induced $\mathbb{Z}$-grading on $E$ of full support of case (B), it remains to consider the situation when $a/d$ and $b/d$ are odd integers, and $d=\gcd(a,b)$. 

\begin{theorem}\label{idG4}
	Suppose that $\{a, b, c\}$ satisfies $\hat{B}$, and assume $a/d$ and $b/d$ are odd integers, where $\gcd(a,b)=d$, and $k\geq d-1$. The $T_\mathbb{Z}$-ideal of the graded identities for $E_{(-b, c, a)}^{(\infty,k,\infty )}$ is generated by the polynomials 
	\begin{enumerate}
			\item[(i)] $[x_{1}, x_{2}]$, if either $\|x_{1}\|$ or $\|x_{2}\|$ is of the form $\alpha d+\beta c$, where $\alpha$ and $\beta$ are both even or both odd;
			\item[(ii)] $x_{1}x_{2}+x_{2}x_{1}$, if both $\|x_{1}\|$ and $\|x_{2}\|$ have the form $\alpha d+\beta c$, where one of $\alpha$ and $\beta$ is even and the other is odd;
			\item[(iii)] $u_1u_2\cdots u_{\alpha}$, where the variables $u_i$ lie in $X_\mathbb{Z}\setminus U$, with $\|u_i\|=q_id+cr_i$ such that $1\leq r_i\leq d-1$, $r_1+\cdots+r_\alpha>k$.
		\end{enumerate}
\end{theorem}
\begin{proof}
First we apply Lemma \ref{A in centro}. If $r=(2{l})d$ then $A_r\subset Z(E)$, and if $r=(2l-1)d$ then $A_r\subseteq E_{(1)}$. Suppose initially that there is an integer $r$ such that $(A_r\cap Z(E))$ and $(A_r\cap E_{(1)})$ are non-trivial. Suppose that $m_1=u_1\cdots u_t e_{i_1}e_{i_2}\cdots e_{i_n}\in A_r\cap Z(E)$ such that each $u_i\in L_{a}^{\infty}\cup L_{-b}^{\infty}$ and $e_{i_l}\in L_{c}^{k}$. Similarly, we consider $m_1=u^\prime_1\cdots u^\prime_k e_{j_1}e_{j_2}\cdots e_{j_m}\in A_r\cap E_{(1)}$. Without loss of generality, assume $m\geq n$. In this case, $u_1\cdots u_t\in A_{r-nc}$, and $u^\prime_1\cdots u^\prime_k e_{j_1}e_{j_2}\cdots e_{j_{m-n}} \in A_{r-nc}$. It follows that there exists $q\in\mathbb{Z}$ such that $qd=r-nc$, and therefore $A_{qd}\cap Z(E)$ and $A_{qd}\cap E_1$ are non-trivial, thus contradicting Lemma \ref{A in centro}. Therefore either $A_r\subseteq Z(E)$ or $A_r\subseteq E_{(1)}$.

Notice that $\gcd(d,c)=1$ implies that $\bar{c}$ is a generator of the group $\mathbb{Z}_d$. It follows that the set $\{0, c,2c,\ldots, (d-1)c\}$ is formed, modulo $d$, by distinct elements. The theorem is proved.
\end{proof}

The last two theorems conclude the complete description of the bases of the graded identities for every $3$-induced $\mathbb{Z}$-grading on $E$ which is of full support.

Let $\mathfrak{V}_{1,k}^\mathbb{Z}$ be the variety  of $\mathbb{Z}$-graded algebras defined by $E_{(-b, c, a)}^{(\infty,k,\infty )}$ such that the product $(a/d) (b/d)$ is even. Moreover, if $(a/d) (b/d)$ is odd, let $\mathfrak{V}_{2,k}^\mathbb{Z}$ denote the variety  of $\mathbb{Z}$-graded algebras defined by $E_{(-b, c, a)}^{(\infty,k,\infty )}$.

We give below a table that contains the data obtained in this section. Let  $d=\gcd(a,b)$, then

\medskip
\begin{tabular}{|c|c|c|c|c|}
\hline
\multirow{5}{*}{$A$}  & \multirow{2}{*}{$k<\infty$} & $a b$ odd and $c$ even & $E_{(-b, c, a)}^{(\infty,k,\infty )}$ generates $\mathfrak{V}_{(-1, 0, 1)}^{(\infty, k,\infty )}$\\ 
\cline{3-3}\cline{4-4}
 & & \multirow{2}{*}{$a b$ even }& \multirow{2}{*}{$E_{(-b, c, a)}^{(\infty,k,\infty )}$ generates $\mathfrak{V}_{(-1, 0, 2)}^{(\infty, k,\infty )}$}\\
\cline{2-2}
 & $k=\infty$ & & \\
\cline{3-3}\cline{2-2}\cline{4-4}
 & $k\geq 0$ & \multirow{2}{*}{$a bc$ odd}& \multirow{2}{*}{$E_{(-b, c, a)}^{(\infty,k,\infty )}$ generates $\mathfrak{V}_{can}$}\\
 \cline{1-2}
\multirow{5}{*}{$B$}   & \multirow{2}{*}{$k=\infty$} & & \\
\cline{3-4}
&  & $a b c$ even& $E_{(-b, c, a)}^{(\infty,k,\infty )}$ generates $\mathfrak{V}_{(-1, 0, 2)}^{(\infty, k,\infty )}$\\
\cline{2-4}
 & \multirow{2}{*}{$k<\infty$} & $(a/d)(b/d)$ even& $E_{(-b, c, a)}^{(\infty,k,\infty )}$ generates $\mathfrak{V}_{1,k}^\mathbb{Z}$\\
\cline{3-4}
 & & $(a/d)(b/d)$ odd& $E_{(-b, c, a)}^{(\infty,k,\infty )}$ generates $\mathfrak{V}_{2,k}^\mathbb{Z}$\\
\hline
\end{tabular}

\section{Generalities}

In this section we comment on some natural generalizations of the results of Section \ref{sect4}. We start with general $n$-induced $\mathbb{Z}$-gradings on the Grassmann algebra and, afterwards, we prove that our results hold for $3$-induced $\mathbb{Z}$-gradings on $E$ such that $S_{(-b,c, a)}^{(\infty,k,\infty)}$ is a subgroup of $\mathbb{Z}$.

\subsection{$n$-induced gradings}
Now we explicit some $n$-induced gradings on $E$ which are of full support. 

\begin{proposition}
	Let $E_{(r_{1},\ldots, r_{n})}^{(v_{1},\ldots, v_{n})}$ be $n$-induced $\mathbb{Z}$-grading on $E$. If there exist $i$, $j\in\{1,\dots, n\}$ such that $r_{i}<0<r_{j}$, $\gcd(r_{i},r_{j})=1$, and $v_{i}=v_{j}=\infty$, then $E_{(r_{1},\ldots, r_{n})}^{(v_{1},\ldots, v_{n})}$ is of full support.
\end{proposition} 
\begin{proof}
The proof is straightforward, hence we omit it.
\end{proof}

Actually, the previous proposition holds in a more general situation. Let $S_{(r_{1},\ldots, r_{n})}^{(v_{1},\ldots, v_{n})}=\langle d\rangle$ and assume that $r_1<r_2<\ldots <r_n$ (this can always be achieved by permuting simultaneously the lower and upper indices). 
	If $r_1>0$, then it is clear that $S_{(r_{1}, \ldots, r_{n})}^{(v_{1}, \ldots, v_{n})}\subset \mathbb{N}_{0}$. If $r_n< 0$ we conclude that $S_{(r_{1},\ldots, r_{n})}^{(v_{1},\ldots, v_{n})}\subset \{-k\mid k\in\mathbb{N}_{0}\}$. 
	Hence it follows that $r_1<0<r_n$. We fix the unique integer $i\in \{1,2,\ldots,n\}$ such that 
\[
r_1<r_2<\ldots<r_i<0<r_{i+1}<\ldots<r_{n}.
\]
	If $v_j=\dim L_{r_j}^{v_j}<\infty$, for $j=i+1$, \dots, $n$, then  $S_{(r_{1},\ldots, r_{n})}^{(v_{1},\ldots, v_{n})}\subset \{a\in\mathbb{Z}\mid a\geq r_1v_1+r_2v_2+\cdots+r_iv_i\}$. A similar conclusion holds if  $v_j=\dim L_{r_j}^{v_j}<\infty$, for  $j=1$, 2, \dots, $i$. We summarize these comments in the following proposition.
	\begin{proposition}
			Let $E_{(r_{1},\ldots, r_{n})}^{(v_{1},\ldots, v_{n})}$ be an $n$-induced $\mathbb{Z}$-grading on $E$. If $S_{(r_{1},\ldots, r_{n})}^{(v_{1},\ldots, v_{n})}$ is a subgroup of $\mathbb{Z}$, then there exist $i$, $j\in\{1,\ldots, n\}$ such that $r_{i}<0<r_{j}$, and $v_{i}=v_{j}=\infty$.
	\end{proposition}

From now on, we assume the conditions of the previous proposition. In order to simplify the notation, assume $r_{1}<0<r_{2}$, $\gcd(r_{1},r_{2})=d$ and $v_{1}=v_{2}=\infty$. In this case, there exist $\alpha$, $\beta$, $\alpha'$, $\beta'$ in $\mathbb{N}_{0}$ such that
\[d=\alpha r_{1}+\beta r_{2},\quad \textrm{ and }\quad -d=\alpha' r_{1}+\beta' r_{2},\]
according to what was done before Proposition \ref{lemas juntos}. Moreover, we fix $\alpha$, $\beta$, $\alpha'$, $\beta'$ with the property that $\alpha+ \beta$ and $\alpha' + \beta'$ are the least possible. Now we study the parities of $\alpha+ \beta$ and $\alpha' + \beta'$. The description of the graded identities of $E_{(r_{1},\ldots, r_{n})}^{(v_{1},\ldots, v_{n})}$ can be performed using the same arguments as those of Section \ref{sect4}.

As mentioned in the Introduction, the methods we develop here suggest that  an adequate usage of Elementary Number Theory could contribute to the study of gradings and graded identities not only on the Grassmann algebra $E$, but also on other important algebras. Recall that if $d$ is a prime, one can apply our methods to the problems studied in the papers \cite{disil,disilplamen}. In other words, the description of gradings on $E$ whose support is a subgroup of $\mathbb{Z}$, and their graded identities,  is ``reduced'' to: 1) the problem of the description of graded identities on $E$ by a cyclic group, and 2) the usage of techniques derived from number theory.

\subsection{Weak isomorphisms}

Corollary \ref{caracteriz} gives us conditions for $E_{(r_{1},r_{2},r_{3})}^{(v_{1}, v_{2}, v_{3})}$ to be a subgroup of $\mathbb{Z}$. When $S_{(r_{1},r_{2},r_{3})}^{(v_{1}, v_{2}, v_{3})}=\mathbb{Z}$, we described completely the graded identities of $E_{(r_{1},r_{2},r_{3})}^{(v_{1}, v_{2}, v_{3})}$. 

Here we study the case where $S_{(r_{1},\ldots, r_{n})}^{(v_{1},\ldots, v_{n})}=\langle \gcd(r_{1},\ldots, r_{n})\rangle$ is a proper subgroup of $\mathbb{Z}$. We give a relation between $T_{\mathbb{Z}}(E_{(r_{1},r_{2},r_{3})}^{(\infty, k, \infty)})$ and $T_{\mathbb{Z}}(E_{(\frac{r_1}{d},\frac{r_2}{d}, \frac{r_3}{d})}^{(\infty, k, \infty)})$, where $d=\gcd(r_{1},r_{2},r_{3})$. To this, we need the following definition.

\begin{definition}
	Let $A=\oplus_{g\in G}A_g$ and $B=\oplus_{h\in H}B_h$ be algebras graded by the groups $G$ and $H$, respectively. The graded algebras $A$ and $B$ are weakly isomorphic if there exists an isomorphism of groups $\rho:G\rightarrow H$ and an isomorphism of algebras $\varphi:A \rightarrow B$ such that $\varphi (A_g)=B_{\rho(g)}$, for every $g$ in $G$. 
\end{definition}

Let $G$ and $H$ be groups and let $A=\oplus_{g\in G}A_{g}$ and  $B=\oplus_{h\in H}B_h$ be algebras graded by the groups $G$ and $H$, respectively. If $A$ and $B$ are weakly isomorphic with respect to the isomorphism of groups $\rho\colon G\to H$ then a polynomial $f\in F\langle X|G \rangle$ is a graded identity for $A$ if and only if $\Phi (f)$ is a graded identity for $B$, where $\Phi: F\langle X|G \rangle \rightarrow F\langle X|H \rangle$ is the isomorphism such that $\Phi (x_{i}^{g})=x_{i}^{\rho(g)}$. Moreover if $S$ is a basis for the $T_G$-ideal $T_G(A)$ then $\Phi(S)$ is a basis for the $T_H$-ideal $T_H(B)$.

Given $d\in \mathbb{Z}$, we consider the map $\Phi_d\colon F\langle X|\mathbb{Z}\rangle\to F\langle X|\langle d\rangle\rangle$ defined by $\Phi_d(x_i^n)=x_i^{dn}$. If $d\neq 0$, it is clear that the map $\rho\colon \mathbb{Z}\to \langle d\rangle$, given by $\rho(n)=dn$, is an isomorphism of groups.
\begin{lema}\label{good1}
	Let $f\in F\langle X|\mathbb{Z}\rangle$ be a multilinear polynomial. Then $f\in T_{\mathbb{Z}}(E_{\left( \frac{r_{1}}{d},\ldots, \frac{r_{n}}{d}\right) }^{(v_{1},\ldots, v_{n})})$ if and only if $\varPhi_d(f)\in T_{\langle d\rangle}(E_{\left( {r_{1}},\ldots, {r_{n}}\right) }^{(v_{1},\ldots, v_{n})})$.
\end{lema}

Now we define the homomorphism $\Psi_d \colon F\langle X|d\mathbb{Z}\rangle\to F\langle X|\mathbb{Z}\rangle$ given by $\Psi_d(x_i^{dn})=x_i^{dn}$. By using the same argument applied in \cite{alan-brandao -claud}, we obtain the  following. 
\begin{theorem}
	Let $S_{(r_{1},\ldots, r_{n})}^{(v_{1},\ldots, v_{n})}$ be a subgroup of $\mathbb{Z}$ and let $d=\gcd(r_{1},\ldots, r_{n})$. If $S$ is a basis for the $T_\mathbb{Z}$-ideal $T_{\mathbb{Z}}(E_{\left( \frac{r_{1}}{d},\ldots, \frac{r_{n}}{d}\right) }^{(v_{1},\ldots, v_{n})})$, then the $T_\mathbb{Z}$-ideal $T_{\mathbb{Z}}(E_{(r_{1},\ldots, r_{n})}^{(v_{1},\ldots, v_{n})})$ is generated by the set $S^\prime\cup N$, where 
	$$S^\prime=\{\Psi_d(\varPhi_d(f))\mid f\in S\}$$
	and
	$$N=\{x\in F\langle X|\mathbb{Z}\rangle\mid \|x\|\notin \langle d\rangle\}.$$
\end{theorem}
\begin{proof}
Let $f$ be an element in $S^\prime$. There exists $g\in S$ such that $\Psi_d(\varPhi_d(g))=f$. By Lemma~\ref{good1}, we have $\varPhi_d(g)\in T_{\langle d\rangle}(E_{\left( {r_{1}},\ldots, {r_{n}}\right) }^{(v_{1},\ldots, v_{n})})$. Notice that the elements of $S_{\left( {r_{1}},\ldots, {r_{n}}\right) }^{(v_{1},\ldots, v_{n})}$ in $\langle d\rangle$ can be seen in $\mathbb{Z}$, and hence we have that the elements of degree $i$ in $S_{\left( {r_{1}},\ldots, {r_{n}}\right) }^{(v_{1},\ldots, v_{n})}$ (in $\langle d\rangle$, remain of degree $i$ in $\mathbb{Z}$. We conclude that $\Psi_d(\varPhi_d(g))\in T_{\mathbb{Z}}(E_{\left( {r_{1}},\ldots, {r_{n}}\right) }^{(v_{1},\ldots, v_{n})})$. Moreover it is clear that $N \subseteq T_{\mathbb{Z}}(E_{\left( {r_{1}},\ldots, {r_{n}}\right) }^{(v_{1},\ldots, v_{n})})$. Hence $\langle S^\prime\cup N\rangle_{T_{\mathbb{Z}}}\subseteq T_{\mathbb{Z}}(E_{\left( {r_{1}},\ldots, {r_{n}}\right) }^{(v_{1},\ldots, v_{n})})$.

Now let $f=f(x_1^{l_1},\ldots, x_r^{l_r})\in T_{\mathbb{Z}}(E_{\left( {r_{1}},\ldots, {r_{n}}\right) }^{(v_{1},\ldots, v_{n})})$ be a multilinear polynomial. If $l_i\in \mathbb{Z}\setminus S_{\left( {r_{1}},\ldots, {r_{n}}\right) }^{(v_{1},\ldots, v_{n})}$, for some $i=1$, \dots,  $r$, then $f\in \langle S^\prime\cup N\rangle_{T_{\mathbb{Z}}}$. Hence we can suppose $l_i\in S_{\left( {r_{1}},\ldots, {r_{n}}\right) }^{(v_{1},\ldots, v_{n})}$,  for every $i=1$, \ldots, $r$. In this case, for each $i$, there exists $s_i\in \mathbb{Z}$ such that $l_i=ds_i$. It follows that we can assume, without loss of generality,  $f\in F\langle X|\langle d\rangle\rangle\cap T_{\mathbb{Z}}(E_{\left( {r_{1}},\ldots, {r_{n}}\right) }^{(v_{1},\ldots, v_{n})})=T_{\langle d\rangle}(E_{\left( {r_{1}},\ldots, {r_{n}}\right) }^{(v_{1},\ldots, v_{n})})$. Lemma \ref{good1} implies $f=\varPhi_d(g)$, where $g=f(x_1^{s_1},\ldots, x_r^{s_r})$ and $g\in T_{\mathbb{Z}}(E_{\left( \frac{r_{1}}{d},\ldots, \frac{r_{n}}{d}\right) }^{(v_{1},\ldots, v_{n})})$. As $T_{\mathbb{Z}}(E_{\left( \frac{r_{1}}{d},\ldots, \frac{r_{n}}{d}\right) }^{(v_{1},\ldots, v_{n})})$ is generated by $S$, then there exist $g^j_1$, $g^j_2$, $h^j_1$, \ldots, $h^j_{p_j} \in F\langle X|\mathbb{Z}\rangle$ and $t_1$, \ldots, $t_u\in S$ such that
\[
g=\sum_j g^j_1t_j(h^j_{1},\ldots,h^j_{p_j}) g^j_{2}.
\]
Therefore
\[
f=\Psi_d(\varPhi_d(g))=\sum_j G^j_1[\Psi_d(\varPhi_d(t_j))(H^j_{1},\ldots,H^j_{p_j})] G^j_{2},
\]
	where $G^j_1, G^j_2, H^j_{1}, \ldots, H^j_{p_j}\in F\langle X|\mathbb{Z}\rangle$, $G^j_i=\Psi_d(\varPhi_d(g^j_i))$, and $H^j_k=\Psi_d(\varPhi_d(h^j_k))$, for $i=1$, 2, and $k=1$, \dots, $p_j$. Thus we have that $f\in \langle S^\prime\cup N\rangle_{T_{\mathbb{Z}}}$, and we are done. 
\end{proof}

The last theorem implies that in order to determine the $\mathbb{Z}$-graded identities for an algebra of type $E_{\left( {r_{1}},\ldots, {r_{n}}\right) }^{(v_{1},\ldots, v_{n})}$ where $S_{\left( {r_{1}},\ldots, {r_{n}}\right) }^{(v_{1},\ldots, v_{n})}=\langle d\rangle$, it is sufficient to consider the case in which such a structure is of full support, namely the case $E_{\left( \frac{r_{1}}{d},\ldots, \frac{r_{n}}{d}\right) }^{(v_{1},\ldots, v_{n})}$. 

As an immediate consequence of the previous theorem, we obtain a complete description of the bases of the graded identities for $3$-induced $\mathbb{Z}$-gradings on $E$. For the next result, we write $a'=a/d$, $b'=b/d$, $c'=c/d$, $d=\gcd(a,b,c)$, and we assume $k\geq d^\prime-1$, where $d^\prime=\gcd(a^\prime,b^\prime)$.
\begin{corollary}\label{idquasefull}
	If $S$ is a basis for the $T_\mathbb{Z}$-ideal $T_{\mathbb{Z}}(E_{(-b',c',a')}^{(\infty, k,\infty)})$, then the $T_\mathbb{Z}$-ideal $T_{\mathbb{Z}}(E_{(-b,c,a)}^{(\infty, k,\infty)})$ is generated by the set $S^\prime\cup N$, where 
\[
S^\prime=\{\Psi_d(\varPhi_d(f))\mid f\in S\}, 
\]
	and
\[N=\{x\in F\langle X|\mathbb{Z}\rangle\mid \|x\|\notin \langle d\rangle\}.
\]
\end{corollary}

\end{document}